\documentclass[a4paper,10pt,reqno]{article}
\usepackage{amsthm}
\usepackage{amsmath,amssymb,amsfonts,amsxtra}
\usepackage{epic}
\usepackage{verbatim}
\usepackage[french,english]{babel}
\usepackage[T1]{fontenc}
\usepackage{lmodern}
\usepackage[mathscr]{euscript}
\usepackage{enumerate}
\usepackage{xcolor}
\definecolor{rouge}{rgb}{0.7,0,0}
\definecolor{bleu}{rgb}{0,0,0.7}
\usepackage[colorlinks=true,linkcolor=bleu,urlcolor=violet,citecolor=rouge]{hyperref}
\usepackage{breakurl}
\usepackage{xspace}
\usepackage{multirow}
\usepackage{longtable}

\usepackage[all]{xy}
\makeatletter
\long\def\nnfoottext#1{\insert\footins{\footnotesize
    \interlinepenalty\interfootnotelinepenalty
    \splittopskip\footnotesep
    \splitmaxdepth \dp\strutbox \floatingpenalty \@MM
    \hsize\columnwidth \@parboxrestore
   \edef\@thefnmark{}
   \edef\@currentlabel{}\@makefntext
    {\rule{\z@}{\footnotesep}\ignorespaces
      #1\strut}}}
\makeatother
\newtheorem{thm}{Theorem}%[section]

\newtheorem{prop}[thm]{Proposition}
\newtheorem{lm}[thm]{Lemma}
\newtheorem{cor}[thm]{Corollary}
\newtheorem{claim}[thm]{Claim}

\theoremstyle{definition}

\theoremstyle{remark}

\newtheorem*{mercis}{Acknowledgments}

\numberwithin{remark}{section}
\theoremstyle{plain}

\def\ad{\operatorname {ad}}

\DeclareMathAlphabet{\calptmx}{OMS}{ztmcm}{m}{n}
\newcommand{\K}{{\Bbbk}}
\newcommand{\Z}{\mathbb{Z}}
\newcommand{\N}{\mathbb{N}}

\newcommand{\g}{\mathfrak{g}}
\newcommand{\h}{\mathfrak{h}}

\newcommand{\kk}{\mathfrak{k}}

\newcommand{\rr}{\mathfrak{r}}
\newcommand{\bb}{\mathfrak{b}}

\newcommand{\nf}{\mathfrak{n}}

\newcommand{\wfr}{\mathfrak{w}}

\newcommand{\CC}{\mathfrak{C}}

\newcommand{\gl}{\mathfrak{gl}}
\newcommand{\sld}{\mathfrak{sl}_2}

\newcommand{\NN}{\mathcal{N}}

\newcommand{\pr}{\mathrm{pr}}

\newcommand{\gnq}{>}

\begin{document}

\title{
\vspace{-4cm} Bases fortement nilpotentes et $n$-uplets des alg\`ebres de Borel\\ Very nilpotent basis and $n$-tuples in Borel subalgebras}
\author{{\sc Micha\"el~Bulois}
\thanks{{\url{michael.bulois@univ-angers.fr}}, +33 6 33 12 04 91, LAREMA, 2 boulevard Lavoisier, 49045 Angers CEDEX 1, France.}}
%%%%%%%%%%%%%%%%%%%%%%%%%%%%%%%%%%%%%%%%%%%%%%%%%%%
\date{}

\maketitle

\nnfoottext{Universit\'e d'Angers,  Laboratoire de Math\'ematiques d'Angers (LAREMA), UMR 6093.}
\vspace{-0.5cm}

{\selectlanguage{french}
\begin{abstract}
Une base (d'espace vectoriel) $B$ d'une alg\`ebre de Lie est dite fortement nilpotente si tous les crochets it\'er\'es des \'el\'ements de  $B$ sont nilpotents.
Dans cette note, on d\'emontre une version am\'elior\'ee du th\'eor\`eme d'Engel. On montre qu'une alg\`ebre de Lie admet une base fortement nilpotente si et seulement si c'est une alg\`ebre nilpotente.  
Lorsque  $\g$ est une alg\`ebre de Lie semi-simple, ceci nous permet de d\'efinir un id\'eal de $S((\g^n)^*)^G$ dont l'ensemble alg\`ebrique associ\'e dans $\g^n$ est l'ensemble des $n$-uplets vivants dans une  m\^eme sous-alg\`ebre de Borel. 
\end{abstract}}

\begin{abstract}
A (vector space) basis $B$ of a Lie algebra is said to be very nilpotent if all the iterated brackets of elements  of $B$ are nilpotent.
In this note, we prove a refinement of Engel's Theorem. We show that a Lie algebra has a very nilpotent basis if and only if it is a nilpotent Lie algebra.
When $\g$ is a semisimple Lie algebra, this allows us to define an ideal of $S((\g^n)^*)^G$ whose associated algebraic set in $\g^n$ is the set of $n$-tuples lying in a same Borel subalgebra. 
\end{abstract}

Rubriques: Alg\`ebres de Lie/ G\'eom\'etrie alg\'ebrique

\section{Introduction and Notation}
Let $\g$ be a Lie algebra defined over an algebraically closed field $\K$ of characteristic~$0$.
The adjoint action of $z\in \g$: $x\mapsto[z,x]$ is denoted by $\ad_z\in \gl(\g)$.
If $h\in\g$, we denote the centralizer of $h$ in $\g$ by $\g^h$.  
When $V$ is a vector space over $\K$, and $A\subset V$,  $\langle A\rangle$ stands for the linear subspace spanned by $A$. The symmetric algebra on $V$ is denoted by $S(V)$. Any subset $J\subset S(V^*)$ defines an algebraic subset $\mathcal V(J):=\{x\in V\mid f(x)=0, \,\forall f\in J\}$.

For $n\in \N^*$, let $\mathcal I_n$ be the set of morphisms $\g^n\rightarrow \g$ defined by induction as follows:
\begin{itemize}
\item For $i\in [\![1,n]\!]$, $\big((y_1,\dots y_n)\mapsto y_i\big)\in \mathcal I_n$.
\item If $f,g\in \mathcal I_n$, then $[f,g]:=\big((y_1,\dots y_n)\mapsto [f(y_1,\dots,y_n),g(y_1,\dots,y_n)]\big)\in~\mathcal I_n$.
\end{itemize}
In particular, $\mathcal I(y_1,\dots,y_n):=\{f(y_1,\dots,y_n)\mid f\in \mathcal I_n\}$ is the set of iterated brackets in $y_1,\dots,y_n$. One defines the depth map on $\mathcal I_n$ by induction:
$$ dep\; \big((y_1,\dots y_n)\mapsto y_i\big)=1, \qquad dep\;[f,g]=max\{dep\; f,dep\; g\}+1.$$
We say that $(y_1,\dots,y_n)$ is a \emph{very nilpotent} basis of $\g$ if the following two conditions hold:
\begin{itemize}
\item $(y_1,\dots,y_n)$ is a basis of the vector space $\g$,
\item$\ad_z$ is nilpotent in $\gl(\g)$ for any $z\in \mathcal I(y_1,\dots, y_n)$.
\end{itemize}
The key result of this note is:
\begin{prop}\label{propcentral}
$\g$ has a very nilpotent basis if and only if $\g$ is nilpotent.
\end{prop}

Proposition \ref{propcentral} can be seen as a refinement of Engel's theorem (see, e.g, \cite[19.3.6]{TY}). Its proof is rather technical and is given in section \ref{proofcentral}.

Assume now that $\g$ is semisimple.  Let $G$ be the algebraic adjoint group of $\g$ and let $p_1,\dots, p_d$ be algebraically independant homogeneous generators of $S(\g^*)^G$, the set $G$-invariant elements of $S(\g^*)$. It is well known that $\mathcal V(p_1,\dots, p_d)$ is the nilpotent cone of $\g$ (see, e.g. \cite[\S31]{TY}). Let $J_0$ be the ideal of $S((\g^n)^*)=\bigotimes_n S(\g^*)$ generated by the polynomials $p_i\circ f$ where $f\in \mathcal I_n$. We define $J$ in the same way, adding the constraint $dep\,f\geqslant 2$. We have $J\subset S((\g^n)^*)^G$. We consider the diagonal action of $G$ on $\g^n$. 
In  section \ref{proofcentral2}, we show how Proposition \ref{propcentral} implies the following Proposition.
\begin{prop}\label{propcentral2}
$$\mathcal V(J)=G.(\bb\times\dots\times\bb), \qquad \mathcal V(J_0)=G.(\nf\times\dots\times\nf)$$
where $\bb$ is any Borel subalgebra of $\g$ with nilpotent radical $\nf$.
\end{prop}

The question of finding such ideals arise naturally when one study the diagonal action of $G$ on $\times_{i=1}^n \g$. Indeed, when $n=1$, $G.\nf$ is the nilpotent cone $\mathcal N$. In the $n=2$ case, several authors pointed out nice generalizations of $\mathcal N$. Let us mention the set of \emph{nilpotent pairs} of \cite{Gi}, whose \emph{principal} elements lie in a finite number of orbits under the diagonal action of $G$. Looking at couples of commuting elements, we get the \emph{nilpotent commuting variety} $\CC^{nil}(\g):=\{(x,y)\in \NN\times \NN\mid[x,y]=0\}$ studied in \cite{Ba,Pr} and which has the nice property of being equidimensional. Finally, the \emph{nilpotent bicone}, studied in \cite{ CM}, is the affine subscheme $\mathfrak N\subset\g\times\g$ defined by the polarized polynomials $p_i(x+ty)=0, \forall t\in \K$. Its underlying set consists of pairs whose any linear combination is nilpotent. It is a non-reduced complete intersection, which contains $\CC^{nil}(\g)$ and which has $G.(\nf\times\nf)$ has an irreducible component.

\begin{mercis}
J-Y.~Charbonnel drew the author's attention to the plausibility of Proposition~\ref{propcentral} and to its importance in view of a characterization of $G.(\bb\times\bb)$ as given in Proposition~\ref{propcentral2}. The author thanks him for these remarks and for his encouragements in the writing of this note.
\end{mercis}

\section{Very nilpotent basis}
\label{proofcentral}
The aim of this section is to prove proposition \ref{propcentral}.
As a first step, we assume that $\g$ is semisimple. We are going to prove that $\g$ has no very nilpotent basis, cf. Corollary \ref{cormainstep}.  

First, we have to state some properties of the characteristic grading of a nilpotent element. Let $y$ be a nilpotent element of $\g$ and embed $y$ in a $\sld$-triple $(y,h,f)$. Consider the characteristic grading $$\g=\bigoplus_{i\in \Z} \g(h,i)$$ where $\g(h,i)=\{z\in \g\mid [h,z]=iz\}$ and
denote by $\pr_i$ the projection $\g\rightarrow\g(h,i)$ with respect to this grading. Then $\g(h,0)=\g^{h}$ is a subalgebra of $\g$, reductive in $\g$, i.e. $\ad_{\g(h,0)}(\g)$ is a semisimple representation. 

\begin{lm}\label{parabolic}
\label{reductive}
i)  An element $x\in \bigoplus_{i\geqslant 0} \g(h,i)$ is nilpotent if and only if $\pr_0(x)$ is nilpotent in $\g(h,0)$.\\
ii) $\big\langle[\g(h,i),\g(h,-i)]\big\rangle\subset\g(0,h)$ is Lie subalgebra, reductive in $\g(0,h)$. 
\end{lm} 
\begin{proof}
i) This follows from the fact that $\bigoplus_{i\geqslant 0} \g(h,i)$ is a parabolic subalgebra of $\g$ having $\g(h,0)$ as Levi factor.

ii) If $i=0$, the result is straightforward.
In the following, we assume $i\neq 0$.
Embed  $h$ in a Cartan subalgebra $\h$. This gives rise to a root system $R(\g,\h)\subset\h^*$. 
Choose a fundamental basis $B$ of the root system $R(\g,\h)$ such that $h$ lies in the positive Weyl Chamber associated to $B$.

Let $[x_1,y_1], [x_2,y_2]$ be two elements of $[\g(h,i), \g(h,-i)]$. Then, 
$$\big[[x_1,y_1],[x_2,y_2]\big]=\Big[\big[[x_1,y_1],x_2\big],y_2\Big]+\Big[x_2,\big[[x_1,y_1], y_2\big]\Big].$$
Since $x_2$ and $\big[[x_1,y_1],x_2\big]$ (resp. $y_2$ and $\big[[x_1,y_1],y_2\big]$) are elements of $\g(h,i)$ (resp. $\g(h,-i)$),  the element  $\big[[x_1,y_1],[x_2,y_2]\big]$ belongs to $\langle[\g(h,i),\g(h,-i)]\rangle$.
By linearity, we deduce that $\big\langle[\g(h,i),\g(h,-i)]\big\rangle$ is a Lie subalgebra of $\g(0,h)$.

Write $R_i:=\{\alpha\in R(\g,\h)\mid \alpha(h)=i\}$. Hence $\g(h,i)=\bigoplus_{i\in R_i}\g^{\alpha}$, where $\g^\alpha$ is the root space associated to $\alpha$. 
Then, we see that \begin{equation}\label{eqdeploy}\bigoplus_{R_{i,-i}}\g^{\alpha}\subset\langle[\g(h,i),\g(h,-i)]\rangle\subset \h\oplus\bigoplus_{R_{i,-i}}\g^{\alpha},\end{equation} where $R_{i,-i}=(R_i+R_{-i})\cap R(\g,\h)=(R_i-R_i)\cap R(\g,\h)$. Since $R_{i,-i}=-R_{i,-i}$, $\big\langle[\g(h,i),\g(h,-i)]\big\rangle$ is reductive and it follows from \eqref{eqdeploy} that its central elements are semisimple in $\g$. Hence the result. 
\end{proof}

The main step of the proof of Proposition \ref{propcentral} lies in the following lemma.
\begin{lm}\label{mainstep}
If $\g$ has a very nilpotent basis, then there exists a non-zero proper Lie subalgebra $\kk\subset\g$, reductive in $\g$, and a basis $(x_1, \dots, x_p)$ of $\kk$ such that  $\ad_z$ is nilpotent in $\gl(\g)$ for each $z\in \mathcal I(x_1,\dots, x_p)$.
\end{lm}
 \begin{proof}
Let $(y_1,\dots,y_n)$ be a very nilpotent basis of $\g$. Each $y_i$ is nilpotent.
Embed $y_1$ in a $\sld$-triple $(y_1,h,f)$ and consider the characteristic grading $$\g=\bigoplus_{i\in \Z} \g(h,i).$$
Denote by $i_0$ the highest weight in this decomposition, i.e. $\g(h,i_0)\neq\{0\}$ and $\g(h,i)=\{0\}$ for all $i\gnq i_0$. We set $\kk:=\langle[\g(h,i_0);\g(h,-i_0)]\rangle$. Since $\kk\subset\g(0,h)\neq \g$ is reductive in $\g$ (Lemma \ref{reductive}), there remains to find the basis $(x_1,\dots, x_p)$.

The endomorphism $\ad_{y_1}$ is nilpotent of order $i_0+1$ and $$(\ad_{y_1})^{i_0}:\g\rightarrow \g(h,i_0)$$ is surjective.
%Let $J\subset[\![1,n]\!]$ be the subset of index $j$ 
We define $z_j:=\ad_{y_1}^{i_0}(y_j)$ for $j\in J:=[\![1,n]\!]$.
By construction, $z_j\in \mathcal I(y_1,\dots,y_n)$ and $(z_j)_{j\in J}$ is a family spanning the vector space $\g(h,i_0)$. On the opposite side, we define $y'_k:=\pr_{-i_0}(y_k)$ for $k\in J$. The family $(y'_k)_{k\in J}$ spans the vector space $\g(h,-i_0)$.

Consider now $\ad_{z_j}:\g\rightarrow \bigoplus_{i\geqslant 0} \g(h,i)$ and define $$x_{j,k}:=\pr_0\circ \ad_{z_j}(y_k)=\ad_{z_j}\circ \;\pr_{-i_0}(y_k), \qquad j,k\in J^2.$$
The family $(x_{j,k})_{j,k}$ spans the vector space $\kk\subset\g(h,0)$. The elements $\ad_{z_j}(y_k)$ belong to $\mathcal I(y_1,\dots y_n)$. Hence, it follows from lemma \ref{parabolic} that the elements $x_{j,k}$ and there iterated brackets are nilpotent. In other words, if $(x_1,\dots, x_p)$ is a basis of $\kk$ extracted from $(x_{j,k})_{j,k\in J^2}$, then it fulfills the required properties. 
\end{proof}

\begin{cor}\label{cormainstep}
Let $\g\neq\{0\}$ be a semisimple Lie algebra, then $\g$ has no very nilpotent basis
\end{cor}
\begin{proof}
We argue by induction on $\dim \g$. Assume that there is no semisimple $\wfr\subset \g$ such that $0\neq\dim \wfr<\dim \g$ having a very nilpotent basis. Assume that $\g$ has one. Then, we define the reductive subalgebra $\kk$ equipped with the basis $(x_1,\dots x_p)$ as in Lemma \ref{mainstep}. By hypothesis $\kk$ is not semisimple. Hence $\kk$ has a non-trivial centre whose elements are semisimple. Therefore there must be some $i\in [\![1,p]\!]$ such that $x_i$ is not nilpotent and we get a contradiction.
\end{proof}

We are now ready to finish the {\bf proof of proposition \ref{propcentral}} in the general case. From now on, we forget the semisimplicity assumption on $\g$.

First of all, we note that whenever $\g$ is nilpotent, then any basis of $\g$ is very nilpotent. 

Conversely, assume that $\g$ is any Lie algebra having a very nilpotent basis $(y_1,\dots, y_n)$. Let
 $\rr$ be the radical of $\g$. The algebra $\g/\rr$ is semisimple and we can extract a very nilpotent basis $(x_1,\dots,x_p)$ of $\g/\rr$ from the projection of the elements $y_i$ via $\g\rightarrow \g/\rr$.
It follows from corollary \ref{cormainstep} that $\g/\rr=\{0\}$. In other words $\g$ is solvable.

Then, one may apply Lie's Theorem (see, e.g., \cite[19.4.4]{TY}) . It states that $\ad_{\g}\subset \gl(\g)$ can be seen as a subspace of a set of upper triangular matrices. In fact, the nilpotency condition on the $\ad_{y_i}$ implies that $\ad_{\g}$ is a subspace of a set of strictly upper triangular matrices. In particular, $\g$ is nilpotent. This ends the proof of Proposition \ref{propcentral}.

\section{$n$-tuples lying in a same Borel subalgebra}
\label{proofcentral2}

In this section $\g$ is assumed to be semisimple. Let $\bb$ be a Borel subalgebra of $\g$ and $\nf$ be the nilradical of $\bb$. Define two ideals of $S((\g^n)^*)^G$ by: 
$$J_0:=\big(p_i\circ f\mid i\in [\![1,d]\!], f\in \mathcal I_n\big), \quad 
J:=\big(p_i\circ f\mid i\in [\![1,d]\!], f\in \mathcal I_n, \,dep\; f\geqslant 2\big),$$
where $p_1,\dots, p_d$ are as in the introduction.
\begin{prop}\begin{eqnarray}
&\mathcal V(J_0)=G.(\nf\times\dots\times\nf),&\notag \\
&\mathcal V(J)=G.(\bb\times\dots\times\bb). &\notag
\end{eqnarray}
\end{prop}
\begin{proof}
The inclusions
$\mathcal V(J_0)\supset G.(\nf\times\dots\times\nf)$ and $\mathcal V(J)\supset G.(\bb\times\dots\times\bb)$ are straightforward.

Let us prove the reverse inclusions.
Let $(y_1,\dots y_n)\in \g$ and let $\kk\subset \g$ be the Lie subalgebra generated by the elements $(y_1,\dots,y_n)$. Assume that $(y_1,\dots,y_n)\subset \mathcal V(J_0)$.
Choose a basis of $\kk$, $(z_1,\dots, z_p)\in (\mathcal I(y_1,\dots y_n))^p$. Then $(z_1,\dots, z_p)$ is a very nilpotent basis of $\kk$. It follows from Proposition \ref{propcentral} that $\kk$ is nilpotent. Hence there exists $g\in G$ such that $g.\nf\supset\kk$ and $(y_1,\dots,y_n)\in g.(\nf\times\dots\times\nf)$.

Assume now that $(y_1,\dots,y_n)\in \mathcal V(J)$. Arguing along the same lines, one finds that $[\kk;\kk]$ is nilpotent. Hence $\kk$ is solvable and there exists $g\in G$ such that $g.\bb\supset\kk$.
\end{proof}

The ideals $J_0$ and $J$ are defined by making use of an infinite number of generators. In fact, if one is more careful with the arguments of section \ref{proofcentral}, it is possible to restrict to a finite number. Let us sketch the proof of this.
A  rough estimation shows that, for $(y_1,\dots y_n)\in \g^n $,  the subalgebra $\kk$ generated by $(y_1,\dots y_n)$ is spanned by $\{f(y_1,\dots,y_n)\mid f\in \mathcal I_n,\, dep\;f\leqslant \dim \g\}$ as a vector space. Then, assume that $(z_1,\dots,z_p)$ is a basis of a semisimple Lie algebra $\kk$. One can restrict in the proof of Lemma \ref{mainstep} and Corollary \ref{cormainstep} to the assumption that $f(z_1,\dots,z_p)$ is nilpotent for $dep \;f\leqslant \dim \kk$. Defining
$$J_0:=\big(p_i\circ f\mid dep\; f\leqslant (\dim \g)^2 \big), \quad 
J:=\big(p_i\circ f\mid  2\leqslant dep\; f\leqslant 2(\dim \g)^2\big),$$
we  claim that:
\begin{claim}
$\mathcal V(J')=\mathcal V(J)$ and $\mathcal V(J_0')=\mathcal V(J_0)$.
\end{claim}
Of course, those new sets of generators are far from being minimal and are much too big to be implemented. For instance, one can consider couples of elements in $\sld$. Our upper bound on the depth of elements of $J_0'$ is $(3)^{2}=9$ and the number of elements of $\mathcal I_2$ with depth $9$ is greater than $2^{2^8}\simeq 10^{77}$.

%%%%%%%%%%%%%%%%%%%%%%%%%%%%%%%%%%%%%%%%%%%%%%%%%%%
%\vfill
%%%%%%%%%%%%%%%%%%%%%%%%%%%%%%%%%%%%%%%%%%%%%%%%%%%


\begin{thebibliography}{99}
  \bibitem [Ba] {Ba} V.~Baranovsky, The variety of pairs
      of commuting nilpotent matrices is irreducible,
      \emph{Transform. Groups} {\bf6} (2001), 3-8.
\bibitem[CM]{CM} J.~Y.~Charbonnel and A.~Moreau, Nilpotent bicone and characteristic submodule in a reductive Lie algebra, \emph{Transform. Groups} {\bf 14} (2009), 319-360.  
    \bibitem [Gi]{Gi} V.~Ginzburg, Principal nilpotent pairs in a semisimple Lie algebra I., 
      \emph{Invent. Math.} {\bf 140} (2000), 511-561.

\bibitem [Pr] {Pr} A.~Premet, Nilpotent commuting varieties of reductive Lie algebras, \emph{Invent. Math.} {\bf154} (2003), 653-683.
    \bibitem[TY]{TY} P.~Tauvel and R.~W.~T.~Yu, \emph{Lie
        algebras and algebraic groups}, Springer Monographs
      in Mathematics, Springer-Verlag, 2005.
  \end{thebibliography}
\end{document}